\documentclass{amsart}
\usepackage{amsmath}
\usepackage{amsthm}
\usepackage{amssymb}
\usepackage{enumitem}

\setenumerate{label={\normalfont(\textit{\roman*})}, ref={\textrm{\roman*}},
wide, leftmargin=*, labelwidth=!, labelindent=0pt, 
}
\usepackage{color}
\usepackage{ifthen}
\usepackage{scalerel}
\usepackage{hyperref}
\usepackage{mathrsfs}
\usepackage[mathscr]{euscript}
\usepackage{lineno}

\newtheorem{theorem}{Theorem}[section]
\newtheorem{proposition}[theorem]{Proposition}
\newtheorem{lemma}[theorem]{Lemma}
\newtheorem{corollary}[theorem]{Corollary}

\newtheorem{alphatheorem}{Theorem}

\theoremstyle{definition}
\newtheorem{definition}[theorem]{Definition}
\newtheorem{example}[theorem]{Example}
\newtheorem{remark}[theorem]{Remark}
\newtheorem{conjecture}[theorem]{Conjecture}

\DeclareMathOperator*{\EEE}{\scalerel*{\mathbb{E}}{\textstyle\sum}}

\renewcommand{\Re}{\operatorname{Re}}

\newcommand{\supp}{\operatorname{supp}}

\newcommand{\bbra}[1]{ { \left[ #1 \right] } } 
\newcommand{\fpa}[1]{\left\lVert #1 \right\rVert_{\mathbb{R}/\mathbb{Z}}}

\newcommand{\bra}[1]{\left( #1 \right)}
\newcommand{\braBig}[1]{\Big( #1 \Big)}
\renewcommand{\tilde}{\widetilde}
\newcommand{\abs}[1]{\left|#1\right|}
\newcommand{\set}[2]{\left\{ #1 \ \middle| \ #2 \right\} }
\newcommand{\ceil}[1]{\left\lceil #1 \right\rceil}
\newcommand{\ffrac}[2]{#1/#2}

\newcommand{\e}{\varepsilon}
\renewcommand{\a}{\alpha}
\renewcommand{\b}{\beta}

\newcommand{\NN}{\mathbb{N}}
\newcommand{\QQ}{\mathbb{Q}}
\newcommand{\PP}{\mathbb{P}}
\newcommand{\EE}{\mathbb{E}}
\newcommand{\ZZ}{\mathbb{Z}}
\newcommand{\RR}{\mathbb{R}}
\newcommand{\CC}{\mathbb{C}}
\newcommand{\UU}{\mathbb{U}}

\newcommand{\M}[1]{\mathcal{M}_{#1}}
\newcommand{\A}[1]{\mathcal{A}_{#1}}

\newcommand{\QM}[2]{\mathcal{QM}_{#1}\ifthenelse{\equal{#2}{}}{}{^{\leq #2}}}
\newcommand{\QA}[2]{\mathcal{QA}_{#1}\ifthenelse{\equal{#2}{}}{}{^{\leq #2}}}

\newcommand{\SM}[2]{\mathcal{SM}_{#1}\ifthenelse{\equal{#2}{}}{}{^{\leq #2}}}
\newcommand{\SA}[2]{\mathcal{SA}_{#1}\ifthenelse{\equal{#2}{}}{}{^{\leq #2}}}

\newcommand{\MD}[1]{{\mathcal{M}}_{#1}(\UU)}
\newcommand{\QMD}[2]{{\mathcal{QM}}_{#1}\ifthenelse{\equal{#2}{}}{}{^{\leq #2}}(\UU)}
\newcommand{\SMD}[2]{{\mathcal{SM}}_{#1}\ifthenelse{\equal{#2}{}}{}{^{\leq #2}}(\UU)}

\newcommand{\freq}{\operatorname{freq}}\newcommand{\bb}{\mathbf}	
\newcommand{\la}[2]{{E}_{#1}(#2)}
\newcommand{\lb}[2]{{E}_{#1}^*(#2)}

\renewcommand{\subset}{\subseteq}

\begin{document}

\author[J.\ Konieczny]{Jakub Konieczny}
\address[J.\ Konieczny]{Einstein Institute of Mathematics Edmond J. Safra Campus, The Hebrew University of Jerusalem Givat Ram. Jerusalem, 9190401, Israel}
\address{Faculty of Mathematics and Computer Science, Jagiellonian University in Krak\'{o}w, \L{}ojasiewicza 6, 30-348 Krak\'{o}w, Poland}
\email{jakub.konieczny@gmail.com}

\title[M\"obius orthogonality for $q$-semimultiplicative sequences]{M\"obius orthogonality for\\ $q$-semimultiplicative sequences}

\begin{abstract}
We show that all $q$-semimultiplicative sequences are asymptotically orthogonal to the M\"obius function, thus proving the Sarnak conjecture for this class of sequences. This generalises analogous results for the sum-of-digits function and other digital sequences which follow from previous work of Mauduit and Rivat.
\end{abstract}

\maketitle 

\makeatletter{}\section{Introduction}\label{sec:Intro}

The M\"{o}bius pseudorandomness principle asserts, roughly speaking, that the sign of the M\"{o}bius function $\mu$ should behave like a random sequence of $\pm 1$'s.  (Recall that $\mu(n) = (-1)^k$ if $n$ is a product of $k$ distinct primes, and $\mu(n) = 0$ if $n$ is divisible by a square.) In particular, one expects that that sums involving the M\"{o}bius functions should exhibit a considerable amount of cancellation.

Many important statements in number theory can be construed as special cases of the aforementioned principle. For instance, the prime number theorem is well-known (see e.g.\ \cite{Tenenbaum-book}) to be equivalent to 
\begin{equation}\label{71:00}
	\EEE_{n < N} \mu(n) \to 0 \text{ as } N \to \infty,
\end{equation}
while Riemann Hypothesis is equivalent to the more precise bound
\begin{equation}\label{71:01}
	\EEE_{n < N} \mu(n) = O_{\e}(N^{-1/2+\e})  \text{ for any fixed } \e > 0.
\end{equation}
(Above, $\EE_{n<N}$ is shorthand for $\frac{1}{N} \sum_{n=0}^{N-1}$, and $O_{\e}(X)$ denotes a quantity bounded in absolute value by $cX$ where $c = c(\e)$ is a constant dependent only on $\e$.)

\newcommand{\tempL}{d}
A way to formalise the M\"{o}bius pseudorandomness principle was proposed by Sarnak in \cite{Sarnak-2009}. Let us say that a bounded sequence $f \colon \NN_0 \to \CC$ is \emph{deterministic} if for any fixed $\e > 0$ all the factors of $f$ of length $\tempL$, i.e., all the points $\bra{f(n+j)}_{j=0}^{\tempL-1} \in \CC^{\tempL}$, can be covered by $\exp(o_{\e}(\tempL))$ balls of radius $\e$ as $\tempL \to \infty$. Equivalently, $f$ is deterministic if there exists a topological dynamical system $(X,T)$ with entropy $0$ as well as a point $x \in X$ and a function $F \in C(X)$ such that $f(n) = F(T^n x)$ for all $n \geq 0$.

\begin{conjecture}[\cite{Sarnak-2009, Sarnak-2012}]
Let $f \colon \NN_0 \to \CC$ be a deterministic function. Then
	\begin{equation}\label{eq:Sarnak-dyn}
	\EEE_{n<N} f(n) \mu(n) \to 0 \text{ as } N \to \infty.
	\end{equation}	 
\end{conjecture}

Various special cases of Sarnak conjecture have been proved; in fact, several were already known at the time it was formulated. For the trivial, one-point, dynamical system, Sarnak conjecture is equivalent to the prime number theorem, and for (still rather trivial) finite dynamical systems, it is equivalent to the prime number theorem in arithmetic progressions. The conjecture is known to hold for circle rotations \cite{Davenport-1937} and more generally nilrotations \cite{GreenTao-2012} (both with effective bounds); for horocycle flows on surfaces of constant negative curvature \cite{BourgainSarnakZiegler-2013} and more generally unipotent translations on homogeneous spaces of connected Lie groups \cite{Peckner-2015}; skew-products of tori \cite{LiuSarnak-2015}; some rank-one systems \cite{Bourgain-2013b, ElAbdalaouiLemanczykDeLaRue-2014}; generalised Morse-Kakutani flows \cite{Veech-2017}; certain interval exchanges \cite{FerencziMauduit-2018}; certain continuous extensions of rotations \cite{KulagaPrzymusLemanczyk-2015}; uniquely ergodic models of quasi-discrete spectrum automorphisms \cite{ElAbdalaouiLemanczykDeLaRue-2015}; and dynamical systems generated by automatic sequences and \cite{Mullner-2017}.
See also \cite{SarnakUpis-2015, DrmotaMullnerSpiegelhofer-2017, Fan-2017, Karagulyan-2017, FerencziKulagaPrzymusLemanczyk-2016, ElAbdalaouiKulagaPrzymusLemanczykDeLaRue-2017}.
For extensive introduction to this subject, we refer to the survey paper \cite{FerencziKulagaPrzymusLemanczyk-2017}.

\mbox{}  

In this paper, we are concerned with several classes of sequences defined in terms of digital expansion. Our purpose is to prove Sarnak conjecture for a new class of such sequences and to provide a streamlined treatment for some previously known results. Throughout, $q \geq 2$ denotes the base in which we are working.

A sequence $f \colon \NN_0 \to \UU = \set{z \in \CC}{\abs{z} = 1}$ is \emph{$q$-multiplicative} if for any $n,m \geq 0$ such that $q^l \mid n$ and $q^l > m$ for some $l \geq 0$ it holds that $f(n+m) = f(n) f(m)$. Two basic examples of such sequences are $n \mapsto e(n \alpha)$ and $n \mapsto e(s_q(n)\alpha)$ where $\alpha \in \RR$, $e(t) = e^{2\pi i t}$ and $s_q(n)$ denotes the sum of digits of $n$ in base $q$. 
A sequence $f \colon \NN_0 \to \CC$ is \emph{$q$-automatic} if there exists a finite automaton which produces $f(n)$ given the base-$q$ expansion of $f$ on input. Classical examples of $2$-automatic sequences include the Thue--Morse sequence $t(n) = (-1)^{s_2(n)}$ and the Rudin--Shapiro sequence $r(n) = (-1)^{\freq_2^{11}(n)}$, where $\freq_2^{11}(n)$ denotes the number of times the string $11$ appears in the binary expansion of $n$. 

If follows work of Indlekofer and Katai \cite{IndlekoferKatai-2002} that $q$-multiplicative sequences satisfy Sarnak conjecture (see Section \ref{sec:Outline} for details), and as mentioned before Sarnak conjecture for automatic sequences was proved by M\"{u}llner \cite{MullnerSpiegelhofer-2017}. Green \cite{Green-2012} showed that any sequence computable by a bounded-depth circuit is orthogonal to the M\"{o}bius function, see also \cite{Bourgain-2013}. A similar result for a class of generalised Morse sequences is obtained in \cite{ElAbdalaouiKasjanLemanczyk-2016} (see also \cite{DownarowiczKasjan-2015}). (For relevant definitions, see the introductions of the cited papers).

Mauduit and Rivat \cite{MauduitRivat-2010} show that the sequences  $n \mapsto e(s_q(n)\alpha)$ satisfy a variant of the prime number theorem, which is considerably stronger than M\"{o}bius orthogonality. Namely, for each $\alpha \in \RR$ such that $(q-1)\alpha \in \RR \setminus \ZZ$ there exists $\sigma > 0$ such that
\begin{equation}\label{eq:71:06}
	\EEE_{n < N} e(s_q(n) \alpha) \Lambda(n) = O(N^{-\sigma}),
\end{equation}
where $\Lambda$ denotes the von Mangoldt function given by $\Lambda(n) = \log p$ if $n$ is a power of a prime $p$, and $\Lambda(n) = 0$ otherwise. (For earlier related results, see\cite{DartygeTenenbaum-2005}). In \cite{MauduitRivat-2015}, Mauduit and Rivat prove a variant of the prime number theorem for a wide class of sequences satisfying a carry property, including in particular the Rudin--Shapiro sequence, as well as the sequences $e(\freq_{2}^{11}(n)\alpha)$ for $\alpha \in \RR$. Related results for block-counting sequences are also obtained in \cite{Hanna-2016}.

Let us say that a sequence $f \colon \NN_0 \to \UU$ is $q$-semimultiplicative if $f(0) = 0$ and there exists $r \geq 0$ such that for any $n,m,k \geq 0$ such that $k < q^l$, $q^l \mid m$, $m < q^{l+r}$, $q^{l+r}|n$ for some $l \geq 0$ it holds that $f(n+m+k) = f(n+m)f(m+k)/f(m)$. Any $q$-multiplicative sequence is $q$-semimultiplicative, and so is $e(\freq_{2}^{11}(n)\alpha)$ for any $\alpha \in \RR$.

\begin{alphatheorem}\label{thm:A}
	Let $f \colon \NN_0 \to \UU$ be a $q$-semimultiplicative sequence. Then
	\begin{equation}\label{eq:54:01a}
		\EEE_{n < N} f(n) \mu(n) \to 0 \text{ as $N \to \infty$.}	
	\end{equation}
\end{alphatheorem}

\subsection*{Convention} Throughout this paper, we treat the base $q \geq 2$ as fixed. In particular, we allow all implicit constants to depend on $q$, and the term ``constant'' is synonymous with ``quantity dependent only on $q$''.

\subsection*{Notation} We introduce some slightly non-standard notation to deal with base-$q$ expansions of integers. If $\Sigma$ is a finite alphabet then $\Sigma^* = \bigcup_{l \geq 0} \Sigma^l$ denotes the set of words over $\Sigma$ (including the empty word $\epsilon$). Positive integers can be identified with words over the $\Sigma_q = \{0,1,\dots,q-1\}$ via the usual base-$q$ expansion. If $n \geq 0$ is an integer then $(n)_q \in \Sigma_q^*$ denotes the expansion of $n$ in base $q$ without leading $0$'s (hence, $(0)_q = \epsilon$). Conversely if $u \in \Sigma_q^*$ is a word over $\Sigma_q$ then $[u]_q = \sum_{l=0}^\infty u_l q^l \in \NN_0$ denotes the result of interpreting $u$ as an number written in base $q$. (By convention, if $u \in \Sigma^l_q$ then $u_j = 0$ for $j > l$.)  In particular, $\bbra{(n)_q}_q = n$ for all $n \geq 0$. The support of an integer $n \geq 0$, denoted by $\supp(n)$, is the set of those $j \in \NN_0$ such that $\bra{(n)_q}_j \neq 0$. The restriction of an integer $n \geq 0$ to an index set $I \subset \NN_0$, denoted by $n|_{I}$, is the result of replacing all digits of $n$ at positions outside $I$ with $0$'s. In particular, $[u]_q|_{I} = \sum_{l \in I} u_l q^l$ for all $u \in \Sigma_q^*$ and $\supp(n|_{I}) \subset I$ for all $n \geq 0$. 

The rest of our notation is fairly standard. 
We write $\UU = \set{ z \in \CC}{ \abs{z} = 1}$, $\NN = \{1,2,\dots\}$, $\NN_0 = \NN \cup \{0\}$, and $[N] = \{0,1,\dots,N-1\}$. For $z \in \UU$, $\arg z \in (-\pi, \pi]$ denotes the argument of $z$.
For $x \in \RR$, $\fpa{x} = \min_{n \in \ZZ}\abs{x-n}$ denotes the distance of $x$ from the nearest integer; since $\fpa{x}$ depends only on $x \bmod 1$, we use the same notation for $x \in \RR/\ZZ$. For a finite set $A$, a condition is said to hold for $\e$-almost every $x \in A$ if it is satisfied for all $x \in A'$ for some $A' \subset A$ with $\abs{A \setminus A'} \leq \e \abs{A}$. If additionally $f \colon A \to \CC$ is a complex-valued function then $\EE_{x \in A} f(x)$ is a shorthand for $\frac{1}{\abs{A}} \sum_{x \in A}f(x)$.

We use standard asymptotic notation. In particular, $Y \ll X$ or $Y = O(X)$ mean that there exists a constant $c > 0$ such that $\abs{Y} \leq c X$. If $c$ is additionally allowed to depend on $Z$ we write $Y \ll_Z X$ or $Y = O_Z(X)$. If $X \ll Y \ll X$ we write $X \sim Y$. 	

\subsection*{Acknowledgements} The author is grateful to Tanja Eisner, Aihua Fan, Dominik Kwietniak, Imre K\'{a}tai, Mariusz Lema\'{n}czyk, Christian Mauduit, Tamar Ziegler for helpful discussions. The author is supported by ERC grant ErgComNum 682150.

\makeatletter{}\section{Proof structure}\label{sec:Outline}

We will now briefly describe the outline of the argument and introduce some convenient terminology.
A key tool is the oft-applied criterion first obtained by Katai \cite{Katai-1986} and later rediscovered in a quantitative form by Bourgain, Sarnak and Ziegler \cite{BourgainSarnakZiegler-2013}.

\begin{theorem}[Katai--Bourgain--Sarnak--Ziegler criterion]\label{thm:BSZ}
	Suppose that $f \colon \NN_0 \to \CC$ is a bounded sequence such that for any pair of sufficiently large distinct primes $p,p'$ it holds that 
	\begin{equation}\label{eq:54:02a}
		\EEE_{n < N} f(pn) \bar f(p'n) \to 0 \text{ as $N \to \infty$.}	
	\end{equation}
	Then for any multiplicative sequence $\nu \colon \NN_0 \to \UU$ it holds that
	\begin{equation}\label{eq:54:03a}
		\EEE_{n < N} f(n) \nu(n) \to 0 \text{ as $N \to \infty$.}	
	\end{equation}
\end{theorem}

 We cannot hope to apply this criterion to all $q$-multiplicative sequences (much less to more general classes of sequences). Indeed, the sequence $f(n) = e(n\alpha)$ is $q$-multiplicative for any $\alpha \in \RR$ and any $q \geq 2$, and $f(pn) \bar f(p' n) = e((p-p')n\alpha) = 1$ for all $n \geq 0$ provided that $(p-p') \alpha \in \ZZ$. However, periodic sequences are independently known to be orthogonal to $\mu$. If fact, it will be convenient to introduce a slightly weaker notion. (Note that the following is not equivalent to several other existing notions of almost-periodicity.) 

\begin{definition}\label{def:almost-periodic}
	Let $f \colon \NN_0 \to \CC$ be a bounded sequence. Then $f$ is \emph{almost periodic} if for any $\e > 0$ there exists $Q > 0$ such that for any $N \geq 0$ there exists a sequence $g \colon [N] \to \CC$ with period $Q$ such that $\abs{ f(n) - g(n)} \leq \e$ holds for $\e$-almost all $n \in [N]$. Additionally, $f$ is \emph{$q$-almost periodic} if $Q$ can be chosen to be a power of $q$.
\end{definition}
We stress that the choice of $g$ in the above definition is allowed to depend on $N$. For instance, the sequence $e(\log \log n)$ is almost periodic, but for any periodic sequence $g \colon \NN_0 \to \CC$ the set of $n \geq 0$ such that $\abs{e(\log \log n) - g(n)} > \frac{99}{100}$ has upper asymptotic density $\geq \frac{1}{2}$. 

Definition \ref{def:almost-periodic} is arranged so that M\"{o}bius orthogonality of all almost periodic sequences follows directly from M\"{o}bius orthogonality of all periodic sequences. One can also construe the family of almost periodic sequences, at least on the intuitive level, as the finitary analogue of the $L^2$-closure of the family of periodic sequences.

\begin{lemma}\label{lem:AP=>Sarnak}
	Suppose that $f \colon \NN_0 \to \CC$ is an almost periodic sequence. Then 
	\begin{equation}\label{eq:54:03b}
		\EEE_{n < N} f(n) \mu(n) \to 0 \text{ as $N \to \infty$.}	
	\end{equation}
\end{lemma}
\begin{proof}
	Fix any $\e > 0$. Let $Q = Q(\e)$ be the period appearing in Definition \ref{def:almost-periodic}, and for $N \geq 0$ let $g_N \colon [N] \to \CC$ be a $Q$-periodic sequence such that $\abs{ f(n) - g(n)} \leq \e$ for $\e$-almost all $n \in [N]$. We may assume without loss of generality that $\abs{g_N(n)} \leq 1$ for all $n$, and since $Q$-periodic $1$-bounded sequences are orthogonal to $\mu$ we have
\begin{equation}\label{eq:45:00}
	\limsup_{N \to \infty} \abs{ \EEE_{n < N} f(n) \mu(n) } \leq 2\e + \limsup_{N \to \infty}\abs{ \EEE_{n < N} g_N(n) \mu(n) } = 2\e.
\end{equation}
Letting $\e \to 0$ completes the proof.
\end{proof}

Definition \ref{def:almost-periodic} is robust, in the sense that minor modifications lead to equivalent properties.
For later reference, we record several basic observations concerning almost periodic sequences. 
\begin{remark}\label{remark:AP-robust}
\begin{enumerate}
\item\label{item:64-A} It is sufficient to verify the conditions in Definition for $N$ which are powers of $q$; indeed, if a condition holds for $\e$-almost all $n \in [q^{L}]$ then the same condition holds for $q\e$-almost all $n \in [N]$ for any $N$ with $q^{L-1}\leq N \leq q^L$.
\item\label{item:64-B} It is sufficient to verify the conditions in Definition \ref{def:almost-periodic} for $N$ sufficiently large with respect to $\e$; indeed, if the sequence $f$ has an approximation like above by $Q(\e)$-periodic sequences for all $N \geq N_0(\e)$, then $f$ also has analogous approximation by $Q(\e) N_0(\e)$-periodic sequences for all $N \geq 0$.
\end{enumerate}
\end{remark}

\begin{lemma}\label{lem:almost-periodic-product}
The product of two almost periodic sequences is almost periodic.
\end{lemma}
\begin{proof}
	This follows directly from definition using standard methods.
\end{proof}

Our plan is to show that almost periodic sequences are essentially the only instances when Theorem \ref{thm:BSZ} cannot be applied to prove Theorem \ref{thm:A}. For the purposes of this paper, we shall say that a sequence satisfies the Katai--Bourgain--Sarnak--Ziegler criterion if it satisfies the assumptions of Theorem \ref{thm:BSZ} (note that \cite{BourgainSarnakZiegler-2013} contains criteria with somewhat weaker assumptions).

\begin{alphatheorem}\label{thm:B}
	Let $f \colon \NN_0 \to \UU$ be a $q$-semimultiplicative sequence. Then at least one of the following holds:
\begin{enumerate}
\item\label{item:B1}
The sequence $f$ is almost periodic.
\item\label{item:B2}
The sequence $f$ satisfies the Katai--Bourgain--Sarnak--Ziegler criterion.
\end{enumerate}
\end{alphatheorem}

\begin{proof}[Proof of Theorem \ref{thm:A} assuming Theorem \ref{thm:B}]
	If $f$ is almost periodic, the claim follows from Lemma \ref{lem:AP=>Sarnak}. Otherwise, $f$ satisfies the Katai--Bourgain--Sarnak--Ziegler criterion, and the claim follows from Theorem \ref{thm:BSZ}; .	
\end{proof}

For $q$-multiplicative sequences, a variant of Theorem \ref{thm:B} is essentially due to Indlekofer and Katai, and a corresponding prime number theorem is shown by Martin, Mauduit and Rivat in \cite{MartinMauduitRivat-2014}. The following is a special case of \cite[Theorem 1]{IndlekoferKatai-2002}, modulo some changes in notation and simple reductions.

\begin{theorem}\label{thm:IK}
	Let $f \colon \NN_0 \to \UU$ be a $q$-multiplicative sequence, and let $p,p' \geq 1$ be integers coprime to each other and to $q$. Then the following conditions are equivalent:
	\begin{enumerate}
	\item There exists $k \geq 0$ such that \[\limsup_{N \to \infty} \abs{\EEE_{n < N} f(pq^kn) \bar f(p'q^kn)} > 0;\]
	\item There exists $\gamma \in \QQ$ such that $q^k (p-p')\gamma \in \ZZ$ for some $k \geq 0$ and 
	\[\sum_{l = 0}^{\infty} \sum_{a =0}^{q-1} \Re\bra{1 - f(aq^l)e(-\gamma a q^l)} < \infty.\]
	\end{enumerate}
\end{theorem}

Theorem \ref{thm:B} for $q$-multiplicative sequences follows rather easily from the above statement. We skip some details of the proof, since we will encounter similar arguments in the remainder of the paper.

\begin{proof}[Proof of Theorem \ref{thm:B} for $q$-multiplicative sequences]
	Suppose that a $q$-multiplicative sequence $f \colon \NN_0 \to \UU$ does not satisfy the Katai--Bourgain--Sarnak--Ziegler criterion for some distinct primes $p,p'$ not dividing $q$. Then, by Theorem \ref{thm:IK}, there exists $\gamma \in \QQ$ such that 
	\[\sum_{l = 0}^{\infty} \sum_{a =0}^{q-1} \Re\bra{1 - f(aq^l)e(-\gamma a q^l)} < \infty.\]
Let $\varphi \colon \NN_0 \to \RR$ be such that $f(n) = e(\varphi(n)+\gamma n)$ for $n \geq 0$. It will suffice to show that $e(\varphi(n))$ is almost periodic. It follows by elementary analysis that
	\[\sum_{l = 0}^{\infty} \sum_{a < q} \abs{ \varphi(aq^l)}^2 < \infty.\]
Take any $\e > 0$. It is now not hard to deduce that there exists some $K = K(\e)$ such that for all $L > K$, $\e$-almost all values of $\varphi(n)$ for $n$ with $\supp(n) \subset [K,L)$ lie within a ball of radius $\e$. (One can treat each digit of $n < q^L$ as a separate random variable, uniformly distributed on $\Sigma_q$, and use additivity of variance and Chebyshev bound to obtain a concentration result for $\varphi(n)$ with $\supp(n) \subset [K,L)$; see also Proposition \ref{prop:convergence-criterion}.) Hence, for $K$ and $L$ as above, there exists a sequence $\psi \colon \NN_0 \to \RR$ with period $q^K$ such that $\fpa{ \varphi(n) - \psi(n)} \leq \e$ for $\e$-almost all $n \in [q^L]$. It follows that $e(\varphi(n))$ is $q$-almost periodic (see Remark \ref{remark:AP-robust}).
\end{proof}

Our strategy for the more general class of $q$-semimultiplicative sequences is to mimic the above proof. We face several technical difficulties. In particular, the analogue of Theorem \ref{thm:IK} is not available, and we provide a somewhat weaker result which is sufficient for our applications. 

\makeatletter{}\section{Multiplicativity}\label{sec:Background}

\newcommand{\llambda}{g}

In this section we discuss several different notions related to multiplicative (and additive) behaviour of sequences related to their expansion in a given base $q$. Recall that $n|_{I}$ denotes integer obtained by replacing all digits in the expansion of $n$ outside $I$ with zeros. We say that two sets $I,J \subset \ZZ$ are \emph{separated by gaps of length $> r$} if $\abs{i-j} > r$ for any $i \in I,\ j \in J$.

\begin{definition}\label{def:multi}
Let $f \colon \NN_0 \to G$ be a sequence taking values in an abelian group.\begin{enumerate}
\item\label{item:multi:1} The sequence $f$ is \emph{$q$-multiplicative} if 
\begin{equation}\label{eq:96:11}
f(n+m) = f(n)f(m) 
\end{equation}
for any $n,m \geq 0$ such that such that $0 \leq m < q^l$ and $q^l | n$ for some $l \geq 0$. The set of $q$-multiplicative sequences $\NN_0 \to G$ is denoted by $\M{q}{}(G)$.

\item\label{item:multi:3} The sequence $f$ is \emph{$q$-semimultiplicative} if there exists $r \geq 0$ such that
\begin{equation}\label{eq:96:01}
f(n+m+k) = f(n+m) f(m)^{-1} f(m+k)
\end{equation}
for any $n,m,k \geq 0$ such that $k < q^l$, $q^l \mid m$, $m < q^{l+r}$ and $q^{l+r} | n$ for some $l \geq 0$, and additionally $f(0) = \mathrm{id}_G$. We refer to the least such $r$ as the \emph{gap} of $f$. The set of $q$-semimultiplicative sequences $\NN_0 \to G$ with gap $\leq r$ is denoted by $\SM{q}{r}(G)$ and $\SM{q}{}(G) = \bigcup_{r = 0}^\infty \SM{q}{r}(G)$.

\item\label{item:multi:2} The sequence $f$ is \emph{$q$-quasimultiplicative} if there exists $r \geq 0$ such that 
\begin{equation}\label{eq:96:21}
f(n+m) = f(n)f(m) 
\end{equation}
for any $n,m \geq 0$ such that $m < q^l$ and $q^{l+r} | n$ for some $l \geq 0$. We refer to the least such $r$ as the \emph{gap} of $f$. The set of $q$-quasimultiplicative sequences $\NN_0 \to G$ with gap $\leq r$ is denoted by $\QM{q}{r}(G)$ and $\QM{q}{}(G) = \bigcup_{r = 0}^\infty \QM{q}{r}(G)$.
 
\item\label{item:multi:4} The sequence $f$ is \emph{strongly $q$-multiplicative} if $f$ is $q$-multiplicative and additionally $f(qn) = f(n)$ for all $n \geq 0$. The notions of strongly $q$-semimultiplicative and strongly $q$-quasimultiplicative sequences are defined accordingly.
 
\item\label{item:multi:5} If the group $G$ is written additively, we refer to sequences $f$ satisfying the above conditions as $q$-additive, $q$-semiadditive, etc., instead and write $\A{q}{}(G)$ in place of $\M{q}{}(G)$, $\QA{q}{}(G)$ in place of $\QM{q}{}(G)$, etc.
\end{enumerate}
\end{definition}

It is often more convenient to see $q$-multiplicative functions as defined in terms of coefficients rather than relations.
\begin{lemma}\label{lem:multi}
Let $f \colon \NN_0 \to G$ be a sequence taking values in an abelian group.\begin{enumerate}

\item\label{item:multi2:1} The sequence $f$ belongs to $\M{q}(G)$ if and only if it takes the form 
\begin{equation}\label{eq:96:10}
 f(n) = \prod_{i=0}^\infty \llambda(n|_{\{i\}})
\end{equation}
for some coefficients $\llambda(\cdot)$ with $\llambda(0) = 1$. 

\item\label{item:multi2:3} The sequence $f$ belongs to $\SM{q}{r}(G)$ if and only if it takes form 
\begin{equation}\label{eq:96:30}
f(n) = \prod_{i=0}^\infty \llambda(n|_{[i,i+r]})
\end{equation}
 for some coefficients $\llambda(\cdot)$ with $\llambda(0) = 1$.
 
\item\label{item:multi2:2} The sequence $f$ belongs to $\QM{q}{r}(G)$ if and only if it takes form 
\begin{equation}\label{eq:96:20}
f(n) = \prod_{\operatorname{gap}(I) \leq r} \llambda(n|_{I})
\end{equation}
for some coefficients $\llambda(\cdot)$ with $\llambda(0) = 1$, where the product runs over all finite sets $I \subset \NN_0$ such that for each $i \in I$ either $i = \max I$ or $(i,i+r] \cap I \neq \emptyset$. 

\end{enumerate}
\end{lemma}
\begin{proof}
	It is clear by direct inspection that if $f$ satisfies any of the properties \eqref{item:multi2:1}, \eqref{item:multi2:3} or \eqref{item:multi2:2} above, then it also satisfies the corresponding property in Definition \ref{def:multi}, so it remains to prove the reverse implication. Suppose now that $f$ satisfies one of the conditions \eqref{item:multi:1}, \eqref{item:multi:3} or \eqref{item:multi:2} in Definition \ref{def:multi}. Multiplying $f$ by the appropriate sequence of the form \eqref{eq:96:10},  \eqref{eq:96:30} or  \eqref{eq:96:20} respectively, we may assume that $f(n|_{I}) = \mathrm{id}_G$ for all $n \geq 0$ and all $I$ which, depending on the case under consideration, are singletons for \eqref{item:multi:1}, have no gaps of length $> r$ for \eqref{item:multi:3}, or are intervals of length $r+1$ for \eqref{item:multi:2}. It will suffice to show that $f(n) = \mathrm{id}_G$ for all $n \geq 0$. This follows by induction on the cardinality of $\supp(n)$.	
\end{proof}

To justify the terminology introduced above, we record some relations between the above notions. Let $G$ be an abelian group. Then 
\[ \M{q}{}(G) = \SM{q}{0}(G) = \QM{q}{0}(G),\]
whence the terms introduced in Definition \ref{def:multi} can be construed as refinements of the notion of $q$-multiplicativity. As suggested by the prefixes semi- and quasi-, we have 
\[ \SM{q}{r}(G) \subseteq \QM{q}{r}(G) \text{ for any $r \geq 0$.} \]
Also, as the notation suggests,
\[ \SM{q}{r}(G) \subseteq \SM{q}{r+1}(G) \text{ and } \QM{q}{r}(G) \subseteq \QM{q}{r+1}(G) \text{ for any $r \geq 0$.} \]
Verification of each of these properties is standard; for similar statements, see also \cite{KropfWagner-2017}. In general, each of the inclusions just mentioned is strict.

\begin{example}\label{ex:multi}
The simplest examples of non-trivial $q$-quasimultiplicative sequences are related to counting patterns. For $u \in \Sigma_q^*$, not consisting only of $0$'s, let $\freq_q^u(n)$ denote the number of times $u$ appears in the expansion $(n)_q$ (if $u$ starts with any leading zeros, then treat $(n)_q$ as starting with an infinite string of zeros). 

For $\alpha \in \RR$, the sequence $\varphi(n) = \freq_q^u(n) \alpha$ is strongly $q$-semimultiplicative with gap $ = \abs{u} - 1$, and $f(n) = e(\varphi(n))$ is strongly $q$-quasimultiplicative with gap $\leq \abs{u} - 1$. The classical Thue--Morse and Rudin--Shapiro sequences fall into this category and are given by $e\bra{\freq_2^{1}(n) / 2}$ and $e\bra{\freq_2^{11}(n) / 2}$ respectively. 	
\end{example}

The notion of a (strongly) $q$-multiplicative sequence is well-established and well-studied; investigations into these classes of sequences date back to \cite{Gelfond-1968, Delange-1972}; see \cite{Katai-2002} for a survey of more recent results. Kropf and Wagner \cite{KropfWagner-2017} study strongly $q$-multiplicative sequences; we refer to their paper for more background and examples. We also note a slight terminological difference: Kropf and Wagner use the term $q$-quasimultiplicative sequence for what we call a strongly $q$-quasimultiplicative sequences 
\footnote{To justify the change of terminology, let us note that one would expect that a $q$-multiplicative sequence should be $q$-quasimultiplicative. This is true with the terminology introduced above, but not with the one in \cite{KropfWagner-2017}. Similar remarks are made in \cite[Section 1]{KropfWagner-2017}.}
Generalised Thue--Morse sequences, studied by Drmota and Morgenbesser in \cite{DrmotaMorgenbesser-2012}, taking values in an abelian group are precisely the same as strongly $q$-multiplicative sequences. Strongly $q$-semiadditive sequences taking values in $\RR$ have been studied under the name of digital sequences, see e.g.\ \cite[Sec.\ 3.3]{AlloucheShallit-book}.
A notion analogous to a $q$-multiplicative sequence for Zeckendorf representation is studied in \cite{DrmotaMullnerSpiegelhofer-2017}.

We record some desirable closure properties. The following lemma is a key reason for our interest in $q$-quasimultiplicative sequences.
\begin{lemma} Let $G$ be an abelian group and let $r \geq 0$. 
	\begin{enumerate}
	\item The classes $\M{q}{}(G)$, $\SM{q}{r}(G)$ and $\QM{q}{r}(G)$ are closed under pointwise products, under pointwise conjugation, and under the operation of dilating by the factor of $q$ which takes a function $f$ to the function $\tilde f$ given by $\tilde f(n) = f(qn)$.
	\item The class $\QM{q}{}(G)$ is closed under restriction to linear subsequences. More precisely, if $f \in \QM{q}{r}(G)$ then for any integers $a \geq 1$, $b \geq 0$, the sequence $\tilde f$ given by $\tilde f(n) = f(an+b)$ is $q$-quasimultiplicative with gap $\leq r + O(\log(ab))$.
	\item If $p \geq 2$ is an integer with $\log_q p \in \QQ$ (equivalently, $p$ and $q$ are powers of the same integer) then $\SM{q}{}(G) = \SM{p}{}(G)$ and $\QM{q}{}(G) = \QM{p}{}(G)$.
	\end{enumerate}
\end{lemma}

A crucial difference between $q$-semimultiplicative sequences and $q$-\hskip0pt quasimultiplicative sequences is that the local behaviour determines the global behaviour for the former but not for the latter. As we have already seen in Lemma \ref{lem:multi}, a $q$-semimultiplicative sequence is uniquely defined by its values on $n \geq 0$ such that $\supp(n)$ is contained in an interval of length $\leq r+1$. The following lemma gives a convenient way to reconstruct a $q$-semimultiplicative sequence from local data. 
\begin{lemma}\label{lem:semimulti-and-lambda}
	Let $G$ be an abelian group, and let $I_1,I_2,I_3 \subset \NN_0$ be unions of intervals of length $\geq r$ such that $\NN_0 = I_1 \cup I_2 \cup I_3$ where the union is disjoint. Then for any $f \in \SM{q}{r}(G)$ and any $n \geq 0$ it holds that
	\begin{equation}\label{eq:17:00}
		f(n) = f(n|_{I_1 \cup I_2})f(n|_{I_2 \cup I_3})f(n|_{I_3 \cup I_1}) f(n|_{I_1})^{-1}f(n|_{I_2})^{-1}f(n|_{I_3})^{-1}
	\end{equation}
\end{lemma}
\begin{proof}
	It is enough to check that each of the coefficients $\llambda(m)$ (as in Lemma \ref{lem:multi}) appears with the same multiplicity on both sides of \eqref{eq:17:00}. This follows by direct inspection.
\end{proof}

Depending on the situation, it is natural to work with $q$-additive sequences taking values either in $\RR$ or in $\RR/\ZZ$. Clearly, if $\varphi \colon \NN_0 \to \RR$ is $q$-additive (or semi-, or quasimultiplicative), then so is $\varphi \bmod{1} \colon \NN_0 \to \RR/\ZZ$. The following lemma provides a converse to this statement. As a direct consequence, we have $\QM{q}{r}(\UU) = \set{n \mapsto e(\varphi(n))}{\varphi \in \QA{q}{r}(\RR) }$ (and likewise for $\SM{q}{r}(\UU)$).

\begin{lemma}\label{lem:lift-T-to-R}
	Suppose that $\varphi \in \QA{q}{r}(\RR/\ZZ)$ (resp.\ $\SA{q}{r}(\RR/\ZZ)$). Then there exists a $\tilde \varphi \in \QA{q}{r}(\RR)$ (resp.\ $\SA{q}{r}(\RR)$) such that $\varphi = \tilde \varphi \bmod{1}$. 
\end{lemma}
\begin{proof}
	Let $\gamma(\cdot)$ be the coefficients describing $\varphi$ as in Lemma \ref{lem:multi}. For any $m$ such that $\gamma(m)$ is defined, let $\tilde\gamma(m) \in \RR$ be such that $\gamma(m) = \tilde \gamma(m) \bmod{1}$. It is clear that the sequence $\tilde \varphi \colon \NN_0 \to \RR$ defined by the coefficients $\tilde \gamma(\cdot)$ satisfies all of the required conditions.
\end{proof}

\color{black} 

\makeatletter{}\section{Preliminaries}\label{sec:Lemmas}

In this section we record some elementary lemmas, which will be helpful in the course of the argument. We begin with the following simple result.

\begin{lemma}\label{lem:concentration}
	Let $\a_n \in \RR/\ZZ$ for $n \in [N]$ and suppose that 
	\begin{equation}\label{eq:10:00}
		\abs{ \EEE_{n<N} e(\a_n) } \geq 1-\e.
	\end{equation}
	for some $\e > 0$. 
	Then there exists $\b \in \RR/\ZZ$ such that for any $\delta > 0$,
	\begin{equation}\label{eq:10:00a}
	\fpa{\a_n - \b} \leq \delta \text{ for $({\e}/{8 \delta^2})$-almost all $n \in [N]$}.
	\end{equation}
	In particular, 
	\[
	\max_{n,m < N} \fpa{\alpha_n - \alpha_m} \leq \sqrt{N\e}.
	\]
\end{lemma}
\begin{proof}
	Fix $\b \in \RR/\ZZ$ such that $e(-\b) \EE_{n<N} e(\a_n)$ is real, so \eqref{eq:10:00} can be written as
	\begin{equation}\label{eq:10:01}
		\EEE_{n<N} \braBig{1-\cos\bra{2\pi(\a_n-\b)}} \leq \e.
	\end{equation}
	Note that if $\fpa{\a_n - \b} \geq \delta$ for some $n < N$ then also $\cos(2\pi(\a_n-\b)) \leq 1 - 8\delta^2$. The bound \eqref{eq:10:00a} now follows from \eqref{eq:10:01} by Markov inequality. For the additional statement, put $\delta = \sqrt{N\e}/2$.
\end{proof}

We will need the following construction, which can be viewed as a quantitative variant of the Shadowing Lemma for the $\times q$ map on $\RR/\ZZ$. 
\begin{lemma}\label{lem:shadowing}
	Let $\alpha_i \in \RR/\ZZ$ and $\e_i > 0$ ($i \geq 0$) be sequences such that $\e_i$ is decreasing and $\fpa{\alpha_{i+1} - q \alpha_i} \leq \e_i$ for each $i \geq 0$. Then there exists $\beta \in \RR/\ZZ$ such that $\fpa{ q^i \beta - \alpha_i} \leq \e_i$ for each $i \geq 0$.
\end{lemma}
\begin{proof}
	Construct a sequence of points $\beta_i \in \RR$ as follows. Take $\beta_0 \in [-1/2,1/2)$ with $\beta_0 \bmod 1 = \alpha_0$. Once $\beta_i$ has been constructed, let $\beta_{i+1}$ be such that $q^{i+1}\beta_{i+1} \bmod 1 = \alpha_{i+1}$ and $\abs{\beta_i - \beta_{i+1}}$ is minimal. Then 
	\begin{equation}\label{eq:10:02a}
\abs{\beta_{i+1} - \beta_i} = \fpa{q^{i+1}\beta_{i+1} - q^{i+1}\beta_i}/q^{i+1} \leq \e_{i}/q^{i+1}
	\end{equation}	
so $\beta_{i}$ converges to some $\beta \in \RR$. Moreover, 
	\begin{equation}\label{eq:10:02}
	\fpa{q^i \b - \a_i} \leq q^i \fpa{\b - \b_i} \leq \e_i(q^{-1} + q^{-2} + \dots) \leq \e_i. \qedhere
	\end{equation}	
\end{proof}

For completeness, we prove that $q$-semimultiplicative sequences are deterministic. By contrast, it is not hard to see that a general $q$-quasimultiplicative sequence need not be deterministic. Note that all norms on $\CC^d$ are equivalent, so in the definition of deterministic sequence it does not matter whether we use balls or any other convex shapes.

\begin{lemma}
	Let $f \in \SMD{q}{r}$. Then for any $\e > 0$ and $\tempL \geq 0$, all factors of $f$  length $d$ are covered by $O(q^{2r}d/\e^2)$ boxes in $\CC^{\tempL}$ of sidelength $\e$. In particular, $f$ is deterministic.
\end{lemma}
\begin{proof}
	Let us say that a factor of length $d$ is \emph{good} if it appears in $f$ at a position divisible by $d$. We first show that for any $l \geq 0$, the good factors of $f$ of length $q^l$ can be covered by $O(q^r/\e)$ boxes of sidelenght $\e$. Recall that for any $n,m$ with $q^l|n$ and $m < q^l$ we have
\[
 f(n+m) = \frac{f(n)}{f(n|_{[l,l+r)})} f(n|_{[l,l+r)} + m).
\]
Hence, any good factor of $f$ of length $q^l$ is contained in one of the $O(q^r/\e)$ boxes of sidelength $\e$ centred around the points $\bra{z f(aq^r + j)}_{j=0}^{q^l-1} \in \CC^{q^l}$, where $0 \leq a < q^r$ and $z$ ranges over some $\e$-dense subset of $\UU$. 

Secondly, note that any factor of $f$ of length $d$ with $d \leq q^l$ can be obtained as a factor of the concatenation of two good factors of $f$ of length $q^l$. Choosing the two good factors independently and choosing any factor of their concatenation of length $d$, we see that all factors of $f$ of length $d$ can be covered by $O(q^{2r}d/\e^2)$ boxes of sidelength $\e$.
\end{proof} 

\makeatletter{}\section{Convergence criterion}\label{sec:Convergence}
In this section we obtain a criterion for Ces\`{a}ro convergence of $q$-\hskip0pt quasimultiplicative sequences to $0$. The problem of Ces\`{a}ro convergence of $q$-multiplicative sequences was studied by Delange \cite{Delange-1972}, who obtained an analogue of Erd\H{o}s--Wintner theorem. In particular, the following is a direct consequence of results contained in \cite{Delange-1972}.

\begin{proposition}
	Let $f \in \MD{q}$. Then the following conditions are equivalent:
	\begin{enumerate}
	\item\label{cond:59:A}
	there exists $k \geq 0$ such that 
	$\displaystyle \limsup_{N \to \infty} \abs{\EEE_{n < N} f(nq^k)} > 0$;
	\item\label{cond:59:B} 
$\displaystyle{\sum_{l = 0}^{\infty} \sum_{a =0}^{q-1} \Re\bra{1 - f(aq^l)} < \infty}$.
	\end{enumerate}
\end{proposition}

Analogous conditions appear in \cite{IndlekoferKatai-2001-AMH} and \cite{IndlekoferKatai-2002} in the context of Ces\`{a}ro convergence of dilated products of $q$-multiplicative sequences, recall Theorem \ref{thm:IK}.

We will need to have quantitative control over how far a given sequence is from the constant sequence $1$ when evaluated on integers with non-zero entries at specified places. The quantity given by the following definition turns out to be a convenient way of recording this information.

\begin{definition}
	Let $f \colon \NN_0 \to \UU$ and let $I \subset \NN_0$ be a finite set. Then $\la{f}{I} \in \RR_{\geq 0} \cup \{+\infty\}$ denotes the unique number such that
	\begin{equation}\label{eq:def-epsilon}
		\abs{ \EEE_{ \sup(n) \subset I} f(n) } = \exp\bra{-\la{f}{I}},
	\end{equation}
	where the average is taken over all $n \geq 0$ whose base-$q$ expansions contain non-zero digits only on positions in $I$. (By convention, $\la{f}{\emptyset} = 1$.) Additionally, $\lb{f}{I} = \min(\la{f}{I},1)$.
\end{definition}
For comparison with other results, note that for $f \colon \NN_0 \to \UU$ and $l \geq 0$ we have
\[
	\lb{f}{\{l\}} \sim \inf_{\beta} \EEE_{a < q} \Re\bra{ 1 - f(aq^l)e(-\beta) }.
\]
We will almost exclusively limit our attention to the case when $I = [K,L)$ is an interval, in which case the average in \eqref{eq:def-epsilon} runs over all $n$ with $q^K \mid n$ and $n < q^L$. For $f \in \QMD{q}{}$, one can construe $\la{f}{I}$ as the contribution to cancellation of Ces\`{a}ro averages of $f$ corresponding to the indices in $I$. This point of view is justified by the following observation. 

\begin{lemma}\label{lem:global>local}
	Let $f \in \QMD{q}{r}$ and let $(I_i)_{i=1}^s$ be a sequence of disjoint intervals contained in an interval $I \subset \NN_0$. Then 
\begin{equation}\label{eq:54:11a}
	\sum_{i=1}^s \lb{f}{I_i} \ll_r \la{f}{I}.
\end{equation}	
In particular, if there exists a sequence $(I_i)_{i=1}^\infty$ of disjoint intervals such that the sum $\sum_{i=1}^\infty \lb{f}{I_i}$ diverges then $\displaystyle \EEE_{n<q^L} f(n) \to 0$ as $L \to \infty$.
\end{lemma}
\begin{proof}
	Partitioning $(I_i)_{i=1}^s$ into $O_r(1)$ parts if necessary, we may assume that the intervals $I_i$ are separated by gaps of length $> 2r$ from each other and by gaps of length $> r$ zeros from the endpoints of $I$. Let $I_i = [K_i,L_i)$, $l_i = \abs{I_i} = K_i-L_i$, and put also $J_i = [K_i-r,K_i) \cup [L_i,L_i+r)$ for $1 \leq i \leq s$. We may also assume, shifting if necessary, that $I = [0,L)$.
	
	Consider a random variable $\bb n$, distributed uniformly on $[q^L]$. Let $\bb Z \subset \NN$ be the (random) set consisting of those $1 \leq i \leq s$ such that $\bb n|_{J_i} = 0$, or, in plainer terms, such that $(\bb n)_q$ has zeroes on all of the positions at distance between $1$ and $r$ from $I_i$. The events $i \in \bb Z$ are mutually independent and $\PP[i \in \bb Z] = 1/q^{2r} =: \rho$ for $1 \leq i \leq s$. 

	For each $1 \leq i \leq s$, put $\bb n_i = \bb n|_{I_i}$ if $i \in \bb Z$ and $\bb n_i = 0$ otherwise, and let $\bb m = \bb n - \sum_{i=1}^s \bb n_i$. This is set up so that  
\begin{equation}\label{eq:42:00a}
	f(\bb n) = f(\bb m) \prod_{i \in \bb Z} f(\bb n_i),
\end{equation} 
and conditional on the value of $\bb Z$, the variables $\bb m$ and $\bb n_i$ ($1 \leq i \leq s$) are mutually independent. Moreover, conditional on $\bb Z$, for each $1 \leq i \leq s$ the variable $\bb n_i$ is either uniformly distributed on $\set{ n \in \NN_0 }{\supp(n) \subset I_i}$ (which is the case if $i \in \bb Z$), or identically $0$ (if $i \not \in \bb Z$).  It follows that
\begin{equation}\label{eq:42:01a}
	\abs{ \EE \bbra{ f(\bb n) \mid \bb Z } } 
	\leq \prod_{i = 1}^s \abs{ \EE\bbra{ f(\bb n_i) \mid \bb Z} } = \exp\bra{ - \sum_{i \in \bb Z} \la{f}{I_i} }
\end{equation}
Using independence of the events $i \in \bb Z$ for $1 \leq i \leq s$, the triangle inequality and the basic inequality $1 -x \leq \exp(-x) \leq 1-x/10$ for $0 \leq x \leq 1$, we now obtain
\begin{equation}\label{eq:42:11a}
	\abs{ \EEE \bbra{ f(\bb n) } } 
	\leq \prod_{i = 1}^s \braBig{ (1-\rho) + \rho \exp(-  \la{f}{I_i} ) } \leq \exp \bra{ - \frac{\rho}{10} \sum_{i=1}^s \lb{f}{I_i}}.
\end{equation}
It follows that 
\begin{equation}\label{eq:42:12a}
	\la{f}{I} \geq \frac{\rho}{10} \sum_{i=1}^s \lb{f}{I_i},
\end{equation}
which is precisely \eqref{eq:54:11a}.

The additional part of the statement follows directly from how $\la{f}{[0,L)}$ is defined and \eqref{eq:54:11a}.
\end{proof}

We are now ready to state and prove the characterisation mentioned above, which is the main result of this section. Similar results for some special classes of $q$-quasimultiplicative sequences can be found in \cite{IndlekoferKatai-2001-AMH, IndlekoferKatai-2002}.

\begin{proposition}\label{prop:convergence-criterion}
	Let $f \in \QMD{q}{}$. Then the following conditions are equivalent:
	\begin{enumerate}
	\item\label{cond:54:A} 	there exists $k \geq 0$ such that 
	$\displaystyle \limsup_{N \to \infty} \abs{\EEE_{n < N} f(nq^k)} > 0$;
	\item\label{cond:54:B} if $(I_i)_{i=1}^\infty$ is a sequence of disjoint intervals then $\displaystyle \sum_{i=1}^\infty \lb{f}{I_i} < \infty$;
	\item\label{cond:54:C} the sequence $f$ is $q$-almost periodic.
	\end{enumerate}
\end{proposition}
\begin{proof}\color{black}
	We will show the chain of implications $\text{\eqref{cond:54:A}}  \implies \text{\eqref{cond:54:B}} \implies   \text{\eqref{cond:54:C}} \implies 
	\text{\eqref{cond:54:A}}$. Fix $r \geq 0$ such that $f \in \QMD{q}{r}$.

	Suppose first that \eqref{cond:54:A} holds for some $k \geq 0$; replacing $f$ with its dilation if necessary, we may assume without loss of generality that $k = 0$.
	For any $N \geq 0$ and $\e > 0$ it is possible to cover $[N]$ with a union of $O_\e(1)$ disjoint intervals of the form $m q^{L+r} + [q^L]$ (with $m \geq 0$ and $L \geq \log_q(N) - O_\e(1)$) and a remainder set of size $\leq \e N$. The average of $f$ over any such interval $m q^{L+r} + [q^L]$ is the same in absolute value as the average over $[q^L]$, whence	
	\[
	0 < \limsup_{N \to \infty} \abs{\EEE_{n < N} f(n)} \leq O_{\e}\bra{ \limsup_{L \to \infty} \abs{\EEE_{n < q^L} f(n)} } + O(\e).
	\]
Taking sufficiently small $\e$, we conclude that 
\[
	\limsup_{L \to \infty} \abs{\EEE_{n < q^L} f(n)} = \limsup_{L \to \infty} \exp\braBig{ - \la{f}{[0,L)} }> 0,
\]
Hence, by Lemma \ref{lem:global>local} for any sequence of disjoint intervals $(I_i)_{i=1}^\infty$ we have
\[
	\sum_{i=1}^\infty \lb{f}{I_i} \ll_r \liminf_{L \to \infty} \la{f}{[0,L)} < \infty,
\]
which proves \eqref{cond:54:B} (in fact, we obtain a uniform bound on all possible sums in \eqref{cond:54:B}).

Secondly, suppose that $\text{\eqref{cond:54:B}}$ holds. Then for any $\e > 0$ there exist $K = K(\e)$ such that such that $\la{f}{I} \leq \e^3$ for all any interval $I$ with $\min I \geq K$. By Lemma \ref{lem:concentration}, for any such $I$ there exists $z_I \in \UU$ such that $\abs{ f(n) - z_I } \leq \e$ for $\e$-almost all $n \geq 0$ with $\supp(n) \subset I$. Let $M = M(\e)$ be the least integer $\geq K$ such that for $\e$-almost all $n$ with $\supp(n) \subset [K,M)$, the base-$q$ expansion $(n)_q$ contains a string of $r$ consecutive zeros at positions in $[K,M)$. For $L \geq M$ we construct a periodic approximation $g$ of $f$ on $[q^L]$ as follows. Let $n < q^L$; if $(n)_q$ contains $r$ consecutive zeros at positions in $[K,M)$ then pick the least index $l \geq K$ such that $n|_{[l,l+r)} = 0$ and set $g(n) = f(n|_{[0,l)}) z_{[l+r,L)}$; otherwise set $g(n) = 1$.
It follows directly from the construction that $g$ is $q^M$-periodic, and that $\abs{f(n) - g(n)} \leq \e$ for $2\e$-almost all $n  < q^L$. Since $\e > 0$ was arbitrary, $f$ is $q$-almost periodic and $\text{\eqref{cond:54:C}}$ follows.

Lastly, suppose that $\text{\eqref{cond:54:C}}$ holds and for the sake of contradiction assume that $\text{\eqref{cond:54:A}}$ is false. Then there exists $L \geq 0$ such that for any $N \geq 0$ there exists a $q^L$-periodic sequence $g_N \colon [N] \to \CC$ such that $\abs{f(n) - g_N(n)} < q^{-r}/10$ for $q^{-r}/10$-almost all $n < N$. Pick $N$ of the form $N = q^L N'$ with $N'$ large enough that $\abs{ \EEE_{n < N'} f(nq^L) } \leq 1/2$. Then

\begin{align*}
	q^{-r}/5 &\geq \EEE_{n<N} \abs{f(n) - g_N(n) } \geq \EEE_{m < q^L} \abs{\EEE_{n < N'}  f(q^L n + m) - g_N(m) }
	\\& \geq q^{-r}  \EEE_{m < q^{L-r}} \abs{f(m) \EEE_{n < N'}  f(nq^L) - g_N(m) } \geq q^{-r}/2,
\end{align*}	
which is a contradiction. Hence, $\text{\eqref{cond:54:A}}$ follows.
\end{proof}

\makeatletter{}\section{Locally controlled sequences}\label{sec:Local}

In this section, we study $q$-semimultiplicative sequences in more detail.  The key property of these sequences which we exploit in the proof of Theorem \ref{thm:B} is that, conversely to Lemma \ref{lem:global>local}, the contributions $\la{f}{I}$ corresponding to arbitrarily long intervals $I$ are controlled in terms of similar terms corresponding to intervals of bounded length. We make this precise in the Proposition \ref{prop:global<local}, which is the main result of this section. To begin with, we mention two simple lemmas.

\begin{lemma}\label{lem:lambda-alternative}
	Let $\varphi \colon \NN_0 \to \RR$ be any sequence and let $f \colon \NN_0 \to \UU$ be given by $f(n) = e(\varphi(n))$. Then for any interval $I \subset \NN_0$ we have
	\begin{equation}
	\label{eq:69:00}
	\lb{f}{I} \sim \inf_{\b \in \RR} \EEE_{\supp(n) \subset I} \fpa{\varphi(n) - \b}^2.
	\end{equation}
\end{lemma}
\begin{proof}
	It follows from elementary analysis and the definition of $\lb{f}{I}$ that

	\begin{align*}
	\lb{f}{I} \sim 1 - \exp\bra{\la{f}{I}}
	 &= \inf_{\b \in \RR} \EEE_{\supp(n) \subset I} \Re\bra{1 -  e (\varphi(n) - \b)}
	 \\ &\sim \inf_{\b \in \RR} \EEE_{\supp(n) \subset I} \fpa{\varphi(n) - \b}^2.
	  \qedhere
	\end{align*}
\end{proof}

\begin{lemma}\label{lem:lambda-addition}
	Let $f \in \QMD{q}{r}$, let $(I_i)_{i=1}^s$ be a sequence of finite subsets of $\NN_0$ separated by gaps of length $>r$, and put $I = \bigcup_{i=1}^s I_i$. Then
	\begin{equation}
	\label{eq:69:01}
	\la{f}{I}  = \sum_{i=1}^s \la{f}{I_i}.
	\end{equation}	
\end{lemma}
\begin{proof}
	This follows directly from the definition of $\la{f}{\cdot}$ and the observation that
	\[
	\EEE_{\supp(n) \subset I} f(n) = \prod_{i=1}^s \EEE_{\supp(n) \subset I_i} f(n). \qedhere
	\]
\end{proof}

We are now ready to prove the main result in this section.
 
\begin{proposition}\label{prop:global<local}
	Let $f \in \SMD{q}{r}$ and let $I \subset \NN_0$ be an interval. Then there exists a sequence $(I_i)_{i=1}^s$ of disjoint intervals of length $\leq 2r$ separated by gaps $> r$ and contained in $I$ such that 
\begin{equation}\label{eq:54:11}
	 \lb{f}{I} \ll \sum_{i=1}^s \lb{f}{I_i}.
\end{equation}	
\end{proposition}
\begin{proof}
Let $\varphi \in \SA{q}{r}(\RR)$ be such that $f(n) = e(\varphi(n))$ for $n \geq 0$. 
By Lemma \ref{lem:lambda-alternative}
\begin{equation}\label{eq:54:12}
	 \lb{f}{I} \sim \inf_{\b \in \RR} \EEE_{\supp(n) \subset I} \fpa{\varphi(n) - \b}^2.
\end{equation}	
For $j \in \{0,1,2\}$ let $J_j = I \cap \bra{3r \ZZ + [jr,(j+1)r) }$. By Lemma \ref{lem:semimulti-and-lambda}, for any $n \geq 0$ with $\supp(n) \subset I$ we have
\begin{equation}\label{eq:54:13}
	\varphi(n) = \varphi\bra{ n|_{J_1 \cup J_2 }} + \varphi\bra{ n|_{J_2 \cup J_3 }} + \varphi\bra{ n|_{J_3 \cup J_1 }} - \varphi\bra{ n|_{J_1 }} - \varphi\bra{ n|_{J_2 }} - \varphi\bra{ n|_{J_3 }}.
\end{equation}	
Inserting \eqref{eq:54:13} into \eqref{eq:54:12} and using the fact that $\fpa{x+y}^2 \ll \fpa{x}^2 + \fpa{y}^2$ for any $x, y \in \RR$ and using Lemma \ref{lem:lambda-alternative} again, we conclude that
\begin{equation}\label{eq:54:14}
	 \lb{f}{I} \ll \max_{J \in \mathcal{J}} \inf_{\gamma \in \RR} \EEE_{\supp(n) \subset I} \fpa{\varphi(n) - \gamma}^2 \ll \max_{J} \lb{f}{J},
 \end{equation}	
 where the maximum is taken over $\mathcal{J} = \{J_1,J_2,J_3,J_1 \cup J_2, J_2 \cup J_3, J_3 \cup J_1\}$. Any $J \in \mathcal{J}$ can be expressed as a union of disjoint intervals of length $\leq 2r$ separated by gaps of length $> r$, whence \eqref{eq:54:11} follows from Lemma \ref{lem:lambda-addition}.
\end{proof}

\begin{remark}
Proposition \ref{prop:global<local} is the only place on the route to the proof of Theorem \ref{thm:A} in where we use the assumption of $q$-semimultiplicativity. Hence, any $q$-quasimultiplicative which satisfies the conclusion of Proposition \ref{prop:global<local} for some $r \geq 0$ is orthogonal to the M\"{o}bius function. In particular, our argument can be applied to strongly $q$-quasimultiplicative sequences.
\end{remark}

\makeatletter{}\newcommand{\C}{{A}}
\newcommand{\D}{{B}}
\newcommand{\kk}{{k_0}}
\newcommand{\mm}{{k_1}}

\section{Structured sequences}\label{sec:Dichotomy}

Recall that for $\alpha \in \QQ$, the $q$-multiplicative sequence $f(n) = e(\alpha n)$ fails to satisfy the Katai--Bourgain--Sarnak--Ziegler criterion; indeed $f(pn) \bar f(p'n) = 1$ for all $n \geq 0$ if $(p-p')\alpha \in \ZZ$. In this section we study $q$-quasimultiplicative sequences $f$ such that for some distinct integers $p,p'$, $f(pn) \bar f(p'n)$ is very close to $1$ for many $n \geq 0$. We show that any such sequence needs to resemble one of the linear phase functions mentioned above.

\begin{proposition}\label{lem:bilinear-sums}
	For any $\C \geq 0$ there exist $\kk = \kk(A) \geq 0$ such that the following holds.
		Let $\varphi \in \QA{q}{\C}(\RR)$, let $2 \leq p,p' \leq \C$ be distinct integers coprime to $q$, and let $I \subset \NN_0$ be an interval. Define $\psi \in \QA{q}{}(\RR)$ by $\psi(n) = \varphi(pn) - \varphi(p'n)$ for $n \geq 0$ and put $J = [\min I + \kk, \max I - \kk]$. 
	Suppose further that for some $\beta \in \RR$ and $\e>0$ it holds that
	\begin{equation}\label{eq:11:90}
		 \fpa{ \psi(n) - \beta } \leq \e \quad \text{for all $n \geq 0$ with $\supp(n) \subset I$.}
	\end{equation}
	Then there exists $\alpha = \ffrac{a}{b} \in \QQ$ with $b$ divisible by $p-p'$ and coprime to $q$ such that
	\begin{equation}\label{eq:11:90x}
	\fpa{\varphi(n) - \alpha n} \ll_\C s_q(n)\e \quad \text{ for all $n \geq 0$ with $\supp(n) \subset J$.} 
	\end{equation}
\end{proposition}	

\begin{proof}
Fix $\C$, and let $\kk$ be a large constant, to be determined in the course of the argument. We can assume that $\abs{I}$ is large and $\e$ is small in terms of $\C$, since otherwise the claim is trivially true with $\alpha = 0$. Since $\psi(0) = 0$, $\fpa{\b} \leq \e$, so replacing $\e$ with $2\e$ we may assume without loss of generality that $\b = 0$. For concreteness, assume also that $p>p'$. 
\begin{lemma}\label{claim:001}
	For any $\D \in \ZZ$ there exists $\mm = \mm(\D) \geq 0$ such that the following is true.  	Then 
\begin{equation}\label{eq:19:00}
\varphi( n + m ) = \varphi(n) + \varphi(m) + O_{\C,\D}(\e)
\end{equation}
holds for any $m,n \geq 0$ with $\supp(m),\ \supp(n),\ \supp(m+n) \subset [\min I + \mm, \max I - \mm]$ and $m < q^{l+B}$, $q^{l} | n$ for some $l \geq 0$.
	\end{lemma}
\begin{proof}
	Note that \eqref{eq:19:00} is true (even without the error term) if $\D \leq -r$ (so that the digits of $n$ and $m$ are separated by $\geq r$ zeros). We reduce to this case as follows.
	
	Let $t = O_{\C,\D}(1)$ be the least integer such that $(p/p')^t > q^{\D+r}$. Pick $\mm$ large enough that $q^{\mm} \geq q^{r} p^t = O_{\C,\D}(1)$. Because $p$ and $q$ are coprime, we can construct an integer $\tilde m$ with $\supp (\tilde{m}) \subset I$ and $p^t \mid m$ such the digits of $m$ and $\tilde m$ only differ on positions in $[\min I, \min I + \mm -r )$. Likewise, we can construct $\tilde n$ with $\supp (\tilde n) \subset I$ with $p^t \mid \tilde n$, differing from $n$ only on positions in $(\max I - \mm + r, \max I]$. Let $\tilde m' = (p'/p)^t \tilde m$ and $\tilde n' = (p'/p)^t \tilde n$.
	By construction, $\tilde m'$ and $\tilde n'$ are integers and 	
	\[
	\max \supp(\tilde m') \leq \max \supp(\tilde m) - \D - r \leq \min \supp(\tilde n) - r = \min \supp(\tilde n') - r,
	\]
	meaning that the positions where non-zero digits appear in $\tilde m'$ and $\tilde n'$ are separated by a gap of length $> r$. Note also that the positions where $m+n$ and $\tilde m + \tilde n$ differ are also separated by gaps of length $> r$ from $\supp(m+n)$, and likewise for $m$ and $\tilde m$, and for $n$ and $\tilde n$. Applying \eqref{eq:11:90} repeatedly we now obtain
	
\begin{align*}\label{eq:19:01}
\varphi(n+m) - \varphi(n) - \varphi(m)  
&= \varphi(\tilde n + \tilde m) - \varphi(\tilde n) - \varphi(\tilde m) 
\\&= \varphi(\tilde n' + \tilde m') - \varphi(\tilde n') - \varphi(\tilde m') + O_{}(t\e) 
 = O_{\C,\D}(\e).
\end{align*}
which is precisely \eqref{eq:19:00}.
\end{proof}

We will apply the above Lemma with $\D = \max(\ceil{\log_q p}, 2) = O_{\C}(1)$. Assume that $\kk \geq \mm(\D) + B$.
Iterating Lemma \ref{claim:001}, we conclude that for any $n < q^\D$ and $l \in J$ we have
\begin{equation}
\label{eq:19:99}
\varphi(nq^l) = n \varphi(q^l) + O_{\C}(\e).
\end{equation}
It follows from the variant of the shadowing lemma, Lemma \ref{lem:shadowing}, and \eqref{eq:19:99} applied with $n = q$ that there exists $\alpha \in \RR$ such that 
\begin{equation}
\label{eq:19:98}
 \varphi(q^l) = q^l \alpha + O_{\C}(\e) \text{ for all $l \in J$.}
\end{equation}
Writing any $n \geq 0$ with $\supp(n) \subset J$ as a sum of $s_q(n)$ not necessarily distinct powers of $q$ and applying Lemma \ref{claim:001} and \eqref{eq:19:98} repeatedly, we conclude that 
\begin{equation}
\label{eq:19:97}
 \varphi(n) = n \alpha + O_{\C}(s_q(n) \e).
\end{equation}

Next, we address rationality of $\alpha$. For any $l \in J$ we have 
\begin{equation}
\label{eq:19:96}
\fpa{\alpha(p-p')q^l} = \fpa{ \varphi(pq^l) - \varphi(p'q^l)} + O_{\C}(\e) = O_{\C}(\e).
\end{equation}
Applying Lemma \ref{lem:shadowing} again (or reasoning directly), we conclude that
\[
	\fpa{\alpha(p-p')q^{\min J} } = O_{\C}(\e/q^{\max J}).
\] 
This means that there exists $\a'$ with $ \a'(p-p')q^{\min J} \in \ZZ$ such that \eqref{eq:19:97} holds with $\a'$ in place of $\a$. Without loss of generality, we may assume that $\a = \a'$.

Lastly, let $m = O_A(1)$ be large enough that the numerator of $(p-p')/q^m$ in reduced form is coprime to $q$. Then the denominator of $\a q^{\min J + m}$ is coprime to $q$, whence there exists $\a''$ such that the denominator of $\a''$ is divisible by $p-p'$ and coprime to $q$ and \eqref{eq:19:97} holds with $\a''$ in place of $\a$ for $n$ with $\supp(n) \subset [\min J +m, \max J]$. Replacing $\a$ with $\a''$ and $\kk$ with $\kk'' = \kk + m$ finishes the argument.
\end{proof}

\begin{remark}\label{rmrk:a-is-unique}
	With the same notation as in Proposition \ref{lem:bilinear-sums}, there exists $\e_0 = \e_0(A)$ such that the value of $\a$ is unique as long as $\e \leq \e_0$ and $\abs{I} \geq 2\kk+1$.
\end{remark}

Unfortunately, in Proposition \ref{lem:bilinear-sums} we need to make assumptions concerning extreme behaviour of $f$, as opposed to average behaviour. For general $q$-multiplicative sequences, we do not know if an analogue of Proposition \ref{lem:bilinear-sums} is true where maxima are replaced with averages. Of course, the distinction between extreme and average behaviour disappears when restrict our attention to intervals of bounded length. We record this observation in the following corollary.

\begin{corollary}\label{cor:structured}
	Let $\C$ and $\kk$ be as in Proposition \ref{lem:bilinear-sums}. Let $f \in \QM{q}{r}(\UU)$ and let $2 \leq p,p' \leq \C$ be distinct integers coprime to $q$. Define $g \in \QM{q}{}(\UU)$ by $g(n) = f(pn)\bar f(p'n)$ for $n \geq 0$. Then for any intervals $J \subset I \subset \NN_0$ with $\min J - \min I = \max I - \max J = k_0$ there exists $\alpha = a/b \in \QQ$, with denominator dividing $p-p'$ and coprime to $q$ such that for $h \in \QM{q}{r}(\UU)$ given by $h(n) = f(n)e(-\a n)$ for $n \geq 0$ we have
\begin{equation}\label{eq:32:00}
		\lb{h}{J} \ll_{\C} \abs{I}^2 q^{\abs{I}} \lb{g}{I}.
\end{equation}
\end{corollary}
\begin{proof} Pick $\varphi \in \QA{q}{r}(\RR)$ and $\psi \in \QA{q}{}(\RR)$ so that $f(n) = e(\varphi(n))$ and $g(n) = e(\psi(n))$ for $n \geq 0$. Define $\e > 0$ by  
\[
\e = \max_{\supp(n) \subset I} \fpa{\psi(n) - \b},
\]
where $\b \in \RR$ is chosen to minimise the value of $\e$. Then, by Lemma \ref{lem:lambda-alternative},
	\[	\lb{g}{I} \gg \inf_{\gamma \in \RR} \EEE_{\supp(n) \subset I}  \fpa{\psi(n) - \gamma}^2 \geq q^{-\abs{I}} \e^2.
	\]
	Let $\a$ be the value produced by an application of Proposition \ref{lem:bilinear-sums} and let $h$ be defined as above. Then it follows by another application of Lemma \ref{lem:lambda-alternative} that
\[
	\lb{h}{J} \ll \inf_{\gamma \in \RR} \EEE_{\supp(n) \subset I}  \fpa{\varphi(n) - \alpha n - \gamma}^2 \ll_{\C} \abs{I}^2 \e^2 \ll  \abs{I}^2 q^{\abs{I}} \lb{g}{I}. \qedhere
\]
\end{proof}

\color{black} 

\makeatletter{}\section{Proof of the main theorem}

We now have all the ingredients needed to finish our argument.

\begin{proof}[Proof of Theorem \ref{thm:B}]
	Suppose that the sequence $f \in \SMD{q}{r}$ does not satisfy the Katai--Bourgain--Sarnak--Ziegler criterion for some distinct primes $p,p'$ not dividing $q$. Hence, Ces\`{a}ro averages of the sequence $g \in \QMD{q}{}$ given by $g(n) = f(pn) \bar f(p'n)$ for $n \geq 0$ do not converge to $0$. By Proposition \ref{prop:convergence-criterion}, for any collection of disjoint intervals $(I_i)_{i=1}^\infty$ we have
\begin{equation}\label{eq:70:00}
	\sum_{i=1}^\infty \lb{g}{I_i} < \infty.
\end{equation}
In particular, for any $l \geq 0$ it holds that  
\begin{equation}\label{eq:70:01}
\max_{\supp(n) \subset [k,k+l)} \abs{ g(n) - 1} \to 0 \text{ as } k \to \infty.
\end{equation}
It follows from Proposition \ref{lem:bilinear-sums} that there exist $\alpha_{l,k} \in \QQ$ ($k,l \geq 0$) with denominators dividing $p-p'$ and coprime to $q$ such that for each $l \geq 0$,
\begin{equation}\label{eq:70:02}
\max_{\supp(n) \subset [k,k+l)} \abs{ f(n) - e(n \a_{l,k})} \to 0 \text{ as } k \to \infty.
\end{equation}
Since $\alpha_{l,k}$ take only finitely many values (no more than $\abs{p-p'}$) it follows from \eqref{eq:70:02} that for each $l \geq 2$ there is some $\a_{l}$ such that $\a_{l,k} = \a_l$ for $k$ large enough (in terms of $l$). Similarly, there exists $\a$ such that $\a_l = \a$ for all $l \geq 0$. (Recall Remark \ref{rmrk:a-is-unique}.) Replacing $f$ with $f'$ given by $f'(n) = f(n) e(-n\alpha)$, we may assume without loss of generality that $\alpha = 0$. (Recall that the set of almost periodic sequences is closed under products by Lemma \ref{lem:almost-periodic-product}.)

Let $\kk = \kk\bra{\max\bra{r,p,p'}}$ be the constant from Proposition \ref{lem:bilinear-sums}. Consider any sequence of disjoint intervals $(J_i)_{i=1}^\infty$ of length $\leq 2r$, and let $I_i = [\min J_i - \kk, \max J_i + \kk]$ if $\min J_i \geq \kk$ and $I_i = \emptyset$ otherwise. By Corollary \ref{cor:structured}, for sufficiently large $i$ we have 
\[ 
\lb{f}{J_i} \ll_{r,p,p'} \lb{g}{I_i}, 
\]
Since the intervals $(I_i)_{i=1}^\infty$ can be partitioned into $O_{r,p,p'}(1)$ sequences of disjoint collections, we conclude from \eqref{eq:70:00} that 
\begin{equation}\label{eq:70:05}
	\sum_{i=1}^\infty \lb{f}{J_i} \ll_{r,p,p'} \sum_{i=1}^\infty \lb{g}{I_i} + 1 < \infty.
\end{equation}
Consider now any sequence of disjoint intervals $(K_i)_{i=1}^\infty$. By Propositon \ref{prop:global<local}, there exist a sequence $(J_i)_{i=1}^\infty$ of disjoint intervals of length $\leq 2r$ such that 
\begin{equation}\label{eq:70:06}
	\sum_{i=1}^\infty \lb{f}{K_i} \ll \sum_{i=1}^\infty \lb{f}{J_i} < \infty.
\end{equation}
Thus, by Proposition \ref{prop:convergence-criterion}, $f$ is almost periodic (in fact, $q$-almost periodic).
\end{proof}

\bibliographystyle{alphaabbr}
\bibliography{bibliography}

\begin{thebibliography}{eAKPLdlR17}

\bibitem[AS03]{AlloucheShallit-book}
J.-P. Allouche and J.~Shallit.
\newblock {\em Automatic sequences}.
\newblock Cambridge University Press, Cambridge, 2003.
\newblock Theory, applications, generalizations.

\bibitem[Bou13a]{Bourgain-2013}
J.~Bourgain.
\newblock M\"obius-{W}alsh correlation bounds and an estimate of {M}auduit and
  {R}ivat.
\newblock {\em J. Anal. Math.}, 119:147--163, 2013.

\bibitem[Bou13b]{Bourgain-2013b}
J.~Bourgain.
\newblock On the correlation of the {M}oebius function with rank-one systems.
\newblock {\em J. Anal. Math.}, 120:105--130, 2013.

\bibitem[BSZ13]{BourgainSarnakZiegler-2013}
J.~Bourgain, P.~Sarnak, and T.~Ziegler.
\newblock Disjointness of {M}oebius from horocycle flows.
\newblock In {\em From {F}ourier analysis and number theory to {R}adon
  transforms and geometry}, volume~28 of {\em Dev. Math.}, pages 67--83.
  Springer, New York, 2013.

\bibitem[Dav37]{Davenport-1937}
H.~Davenport.
\newblock On some infinite series involving arithmetical functions (ii).
\newblock {\em The Quarterly Journal of Mathematics}, os-8(1):313--320, 1937.

\bibitem[Del72]{Delange-1972}
H.~Delange.
\newblock Sur les fonctions {$q$}-additives ou {$q$} -multiplicatives.
\newblock {\em Acta Arith.}, 21:285--298. (errata insert), 1972.

\bibitem[DK15]{DownarowiczKasjan-2015}
T.~Downarowicz and S.~Kasjan.
\newblock Odometers and {T}oeplitz systems revisited in the context of
  {S}arnak's conjecture.
\newblock {\em Studia Math.}, 229(1):45--72, 2015.

\bibitem[DM12]{DrmotaMorgenbesser-2012}
M.~Drmota and J.~F. Morgenbesser.
\newblock Generalized {T}hue-{M}orse sequences of squares.
\newblock {\em Israel J. Math.}, 190:157--193, 2012.

\bibitem[DMS17]{DrmotaMullnerSpiegelhofer-2017}
M.~Drmota, C.~M{\"u}llner, and L.~Spiegelhofer.
\newblock M\"{o}bius orthogonality for the {Z}eckendorf sum-of-digits function.
\newblock {\em arXiv preprint arXiv:1706.09680}, 2017.

\bibitem[DT05]{DartygeTenenbaum-2005}
C.~Dartyge and G.~Tenenbaum.
\newblock Sommes des chiffres de multiples d'entiers.
\newblock {\em Ann. Inst. Fourier (Grenoble)}, 55(7):2423--2474, 2005.

\bibitem[eAKL16]{ElAbdalaouiKasjanLemanczyk-2016}
E.~H. el~Abdalaoui, S.~Kasjan, and M.~Lema\'nczyk.
\newblock 0-1 sequences of the {T}hue-{M}orse type and {S}arnak's conjecture.
\newblock {\em Proc. Amer. Math. Soc.}, 144(1):161--176, 2016.

\bibitem[eAKPLdlR17]{ElAbdalaouiKulagaPrzymusLemanczykDeLaRue-2017}
E.~H. el~Abdalaoui, J.~Ku\l{}aga-Przymus, M.~Lema\'nczyk, and T.~de~la Rue.
\newblock The {C}howla and the {S}arnak conjectures from ergodic theory point
  of view.
\newblock {\em Discrete Contin. Dyn. Syst.}, 37(6):2899--2944, 2017.

\bibitem[eALdlR14]{ElAbdalaouiLemanczykDeLaRue-2014}
E.~H. el~Abdalaoui, M.~Lema\'nczyk, and T.~de~la Rue.
\newblock On spectral disjointness of powers for rank-one transformations and
  {M}\"obius orthogonality.
\newblock {\em J. Funct. Anal.}, 266(1):284--317, 2014.

\bibitem[eALdlR17]{ElAbdalaouiLemanczykDeLaRue-2015}
E.~H. el~Abdalaoui, M.~Lema\'nczyk, and T.~de~la Rue.
\newblock Automorphisms with quasi-discrete spectrum, multiplicative functions
  and average orthogonality along short intervals.
\newblock {\em Int. Math. Res. Not. IMRN}, (14):4350--4368, 2017.

\bibitem[Fan17]{Fan-2017}
A.~Fan.
\newblock Fully oscillating sequences and weighted multiple ergodic limit.
\newblock {\em C. R. Math. Acad. Sci. Paris}, 355(8):866--870, 2017.

\bibitem[FKPL17]{FerencziKulagaPrzymusLemanczyk-2017}
S.~Ferenczi, J.~Ku\l{}aga-Przymus, and M.~Lemanczyk.
\newblock Sarnak's conjecture--what's new.
\newblock {\em arXiv preprint arXiv:1710.04039}, 2017.

\bibitem[FKPLM16]{FerencziKulagaPrzymusLemanczyk-2016}
S.~Ferenczi, J.~Ku\l{}aga-Przymus, M.~Lema\'nczyk, and C.~Mauduit.
\newblock Substitutions and {M}\"obius disjointness.
\newblock In {\em Ergodic theory, dynamical systems, and the continuing
  influence of {J}ohn {C}. {O}xtoby}, volume 678 of {\em Contemp. Math.}, pages
  151--173. Amer. Math. Soc., Providence, RI, 2016.

\bibitem[FM18]{FerencziMauduit-2018}
S.~Ferenczi and C.~Mauduit.
\newblock On {S}arnak's conjecture and {V}eech's question for interval
  exchanges.
\newblock {\em J. Anal. Math.}, 134(2):545--573, 2018.

\bibitem[Gel68]{Gelfond-1968}
A.~O. Gel'fond.
\newblock Sur les nombres qui ont des propri\'et\'es additives et
  multiplicatives donn\'ees.
\newblock {\em Acta Arith.}, 13:259--265, 1967/1968.

\bibitem[Gre12]{Green-2012}
B.~Green.
\newblock On (not) computing the {M}\"obius function using bounded depth
  circuits.
\newblock {\em Combin. Probab. Comput.}, 21(6):942--951, 2012.

\bibitem[GT12]{GreenTao-2012}
B.~Green and T.~Tao.
\newblock The {M}\"obius function is strongly orthogonal to nilsequences.
\newblock {\em Ann. of Math. (2)}, 175(2):541--566, 2012.

\bibitem[Han16]{Hanna-2016}
G.~Hanna.
\newblock {Blocs de chiffres de taille croissante dans les nombres premiers}.
\newblock 2016.
\newblock Preprint. {\href{https://arxiv.org/abs/1611.10279}{arXiv:1611.10279
  [math.NT]}}.

\bibitem[IK01]{IndlekoferKatai-2001-AMH}
K.-H. Indlekofer and I.~K\'atai.
\newblock Investigations in the theory of {$q$}-additive and
  {$q$}-multiplicative functions. {I}.
\newblock {\em Acta Math. Hungar.}, 91(1-2):53--78, 2001.

\bibitem[IK02]{IndlekoferKatai-2002}
K.-H. Indlekofer and I.~K{\'a}tai.
\newblock Investigations in the theory of {$q$}-additive and
  {$q$}-multiplicative functions. {II}.
\newblock {\em Acta Math. Hungar.}, 97(1-2):97--108, 2002.

\bibitem[Kar17]{Karagulyan-2017}
D.~Karagulyan.
\newblock On {M}\"obius orthogonality for subshifts of finite type with
  positive topological entropy.
\newblock {\em Studia Math.}, 237(3):277--282, 2017.

\bibitem[K{\'a}t86]{Katai-1986}
I.~K{\'a}tai.
\newblock A remark on a theorem of {H}. {D}aboussi.
\newblock {\em Acta Math. Hungar.}, 47(1-2):223--225, 1986.

\bibitem[K{\'a}t02]{Katai-2002}
I.~K{\'a}tai.
\newblock On {$q$}-additive and {$q$}-multiplicative functions.
\newblock In {\em Number theory and discrete mathematics ({C}handigarh, 2000)},
  Trends Math., pages 61--76. Birkh\"auser, Basel, 2002.

\bibitem[KPL15]{KulagaPrzymusLemanczyk-2015}
J.~Ku\l{}aga-Przymus and M.~Lema\'nczyk.
\newblock The {M}\"obius function and continuous extensions of rotations.
\newblock {\em Monatsh. Math.}, 178(4):553--582, 2015.

\bibitem[KW17]{KropfWagner-2017}
S.~Kropf and S.~Wagner.
\newblock On {$q$}-quasiadditive and {$q$}-quasimultiplicative functions.
\newblock {\em Electron. J. Combin.}, 24(1):Paper 1.60, 22, 2017.

\bibitem[LS15]{LiuSarnak-2015}
J.~Liu and P.~Sarnak.
\newblock The {M}\"obius function and distal flows.
\newblock {\em Duke Math. J.}, 164(7):1353--1399, 2015.

\bibitem[MMR14]{MartinMauduitRivat-2014}
B.~Martin, C.~Mauduit, and J.~Rivat.
\newblock Th\'eor\'eme des nombres premiers pour les fonctions digitales.
\newblock {\em Acta Arith.}, 165(1):11--45, 2014.

\bibitem[MR10]{MauduitRivat-2010}
C.~Mauduit and J.~Rivat.
\newblock Sur un probl\`eme de {G}elfond: la somme des chiffres des nombres
  premiers.
\newblock {\em Ann. of Math. (2)}, 171(3):1591--1646, 2010.

\bibitem[MR15]{MauduitRivat-2015}
C.~Mauduit and J.~Rivat.
\newblock Prime numbers along {R}udin-{S}hapiro sequences.
\newblock {\em J. Eur. Math. Soc. (JEMS)}, 17(10):2595--2642, 2015.

\bibitem[MS17]{MullnerSpiegelhofer-2017}
C.~M{\"u}llner and L.~Spiegelhofer.
\newblock Normality of the {T}hue-{M}orse sequence along {P}iatetski-{S}hapiro
  sequences, {II}.
\newblock {\em Israel J. Math.}, 220(2):691--738, 2017.

\bibitem[M{\"u}l17]{Mullner-2017}
C.~M{\"u}llner.
\newblock Automatic sequences fulfill the {S}arnak conjecture.
\newblock {\em Duke Math. J.}, 166(17):3219--3290, 2017.

\bibitem[Pec15]{Peckner-2015}
R.~Peckner.
\newblock {\em Two dynamical perspectives on the randomness of the {M}obius
  function}.
\newblock ProQuest LLC, Ann Arbor, MI, 2015.
\newblock Thesis (Ph.D.)--Princeton University.

\bibitem[Sar09]{Sarnak-2009}
P.~Sarnak.
\newblock Three lectures on the möbius function, randomness and dynamics.,
  2009.

\bibitem[Sar12]{Sarnak-2012}
P.~Sarnak.
\newblock Mobius randomness and dynamics.
\newblock {\em Not. S. Afr. Math. Soc.}, 43(2):89--97, 2012.

\bibitem[SU15]{SarnakUpis-2015}
P.~Sarnak and A.~Ubis.
\newblock The horocycle flow at prime times.
\newblock {\em J. Math. Pures Appl. (9)}, 103(2):575--618, 2015.

\bibitem[Ten15]{Tenenbaum-book}
G.~Tenenbaum.
\newblock {\em Introduction to analytic and probabilistic number theory},
  volume 163 of {\em Graduate Studies in Mathematics}.
\newblock American Mathematical Society, Providence, RI, third edition, 2015.
\newblock Translated from the 2008 French edition by Patrick D. F. Ion.

\bibitem[Vee17]{Veech-2017}
W.~A. Veech.
\newblock M\"obius orthogonality for generalized {M}orse-{K}akutani flows.
\newblock {\em Amer. J. Math.}, 139(5):1157--1203, 2017.

\end{thebibliography}

\end{document}